\documentclass[11pt]{article}

\usepackage{amsfonts}
\usepackage{amsmath}
\usepackage{amssymb}
\usepackage{amscd}
\usepackage{amsthm}

\newtheorem{theorem}{Theorem}[section]
\newtheorem{lemma}[theorem]{Lemma}
\theoremstyle{definition}
\newtheorem {definition}[theorem]{Definition}

\theoremstyle{remark}
\newtheorem{remark}[theorem]{Remark}

\def\ra{\rightarrow}
\def\iy{\infty}

\def\be{\begin{equation}}
\def\ee{\end{equation}}
\def\ba{\begin{eqnarray*}}
\def\ea{\end{eqnarray*}}
\def\bae{\begin{eqnarray}}
\def\eae{\end{eqnarray}}
\def\bc{\begin{center}}
\def\ec{\end{center}}

\begin{document}
\title{Limit Distribution of Eigenvalues for Random Hankel and Toeplitz Band Matrices}
\date{September, 2009}
\author{
Dang-Zheng Liu \ and
Zheng-Dong Wang\\
School of Mathematical Sciences\\
Peking University\\
Beijing, 100871, P. R. China }

\maketitle

\begin{abstract}
Consider real symmetric, complex Hermitian Toeplitz and real
symmetric Hankel band matrix models, where the bandwidth $b_{N}\ra
\iy$ but $b_{N}/N \rightarrow b$, $b\in [0,1]$ as $N\rightarrow
\infty$. We prove  that the distributions of eigenvalues converge
 weakly  to universal, symmetric distributions
$\gamma_{_{T}}(b)$ and $\gamma_{_{H}}(b)$. In the case $b>0$ or
$b=0$ but with the addition of $b_{N}\geq
C\,N^{\frac{1}{2}+\epsilon_{0}}$ for some positive constants
$\epsilon_{0}$ and $C$, we prove
 almost sure convergence. The even moments of these
distributions are the  sum of some integrals related to certain pair
partitions. In particular, when the bandwidth grows slowly, i.e.
$b=0$, $\gamma_{_{T}}(0)$ is the standard Gaussian distribution and
$\gamma_{_{H}}(0)$ is the distribution $|x|\,\exp(-x^{2})$. In
addition, from the fourth moments we know that the
$\gamma_{_{T}}(b)$'s are different for different $b$'s, the
$\gamma_{_{H}}(b)$'s different for different $b\in [0,\frac{1}{2}]$
and the $\gamma_{_{H}}(b)$'s different for different $b\in
[\frac{1}{2},1]$.
\end{abstract}


\section{Introduction}
In Random Matrix Theory, the most important information is contained
in the eigenvalues of matrices and the most prominent analytical
object is the  distribution of eigenvalues. That is,  for a real
symmetric or complex Hermitian $N \times N$ matrix $A$ with
eigenvalues $\lambda_{1}, \cdots, \lambda_{N}$, the distribution of
its eigenvalues is the normalized probability measure
\be\label{ave.distr}\mu_{A}:=\frac{1}{N}\sum_{j=1}^{N}\delta_{\lambda_{j}}.\ee
If $A$ is a random matrix, then $\mu_{A}$ is a random measure. 

The limit distribution of (\ref{ave.distr}) is of much interest in
Random Matrix Theory. In \cite{wigner1}, Wigner found his famous
semicircular law and then in \cite{wigner2} pointed out that the
semicircular law was valid for a wide class of real symmetric random
matrices. Since then much more has been done on various random
matrix models, a standard reference is Metha's book \cite{mehta2}.
Recently in his review paper \cite{bai}, Bai proposes the study of
random matrix models with certain additional linear structure. In
particular, the properties of the distributions of eigenvalues for
random Hankel and Toeplitz matrices with independent entries are
listed among the unsolved random matrix problems posed in
\cite{bai}, Section 6. Bryc, Dembo and Jiang \cite{bdj} proved the
existence of limit distributions $\gamma_{_{H}}$ and $\gamma_{_{T}}$
for real symmetric Hankel and Toeplitz matrices. The moments of
$\gamma_{_{H}}$ and $\gamma_{_{T}}$ are the sum of volumes of
solids, from which we can see that $\gamma_{_{H}}$ and
$\gamma_{_{T}}$ are symmetric and of unbounded support. At the same
time Hammond and Miller \cite{hm} also independently proved the
existence of limit distribution for symmetric Toeplitz matrices. In
the present paper we shall prove the existence of limit distribution
of eigenvalues for real symmetric and complex Hermitian Toeplitz
band matrices, and also for real symmetric Hankel band matrices.
Random band matrices have arisen in connection with the theory of
quantum chaos \cite{cmi,fm} , and the limit distribution of
eigenvalues of band matrices has been studied in \cite{bmp,mpk}.

Note that Toeplitz matrices emerge in many aspects of mathematics
and physics and also in plenty of applications, e.g.\,\cite{bg},
especially \cite{basor,diaconis} for connections with random
matrices. Hankel matrices arise naturally in problems involving
power moments, and are closely related to Toeplitz matrices.
Explicitly, for a Toeplitz matrix of the form
$T_{N}=(a_{i-j})_{i,j=1}^{N}$ and a Hankel matrix of the form
$H_{N}=(a_{i+j-2})_{i,j=1}^{N}$, if
$P_{N}=(\delta_{i-1,N-j})_{i,j=1}^{N}$ the ``backward identity"
permutation, then $P_{N}T_{N}$ is a Hankel matrix  and $P_{N}H_{N}$
is a Toeplitz matrix. In this paper we always write a Hankel matrix
$H_{N}=P_{N}T_{N}$ where we assume the matrix entries $a_{-N+1},
\cdots, a_{0}, \cdots,a_{N-1}$ are real-valued, thus $H_{N}$ is a
real symmetric matrix. In addition, if we introduce the Toeplitz or
Jordan matrices $B=(\delta_{i+1,j})_{i,j=1}^{N}$ and
$F=(\delta_{i,j+1})_{i,j=1}^{N}$, called the ``backward shift" and
``forward shift" because of their effect on the elements of the
standard basis $\{e_{1}, \cdots, e_{N}\}$ of $\mathbb{R}^{N}$, then
an $N \times N$ matrix $A$ can be written in the form \be
A=\sum_{j=0}^{N-1}a_{-j} B^{j}+ \sum_{j=1}^{N-1}a_{j} F^{j}\ee if
and only if $A$ is a Toeplitz matrix where $a_{-N+1}, \cdots, a_{0},
\cdots,a_{N-1}$ are complex numbers \cite{hj}. It is worth
emphasizing that this representation of a Toeplitz matrix is of
vital importance and the starting point
 of our method. The
``shift" matrices $B$ and $F$ exactly present the information of the
traces.

Consider a Toeplitz band matrix as follows. Given a band width
$b_{N}<N$, let \be{\eta_{ij}=
 \begin{cases} 1,\ \ |i-j|\leq b_{N};\\
0,\ \ \text{otherwise.}
\end{cases}
}\ee Then a Toeplitz band matrix is
\be{\label{bandtmatrices}T_{N}=(\eta_{ij}\,
a_{i-j})_{i,j=1}^{N}.}\ee Moreover, the Toeplitz band matrix $T_{N}$
can be also written in the form

\be \label{basicrepresentation}T_{N}=\sum_{j=0}^{b_{N}}a_{-j} B^{j}+
\sum_{j=1}^{b_{N}}a_{j} F^{j}.\ee Obviously, a Toeplitz matrix can
be considered as a band matrix with the bandwidth $b_{N}=N-1$. Note
that when referring to a Hankel band matrix $H_{N}$, we always mean
$H_{N}=P_{N}T_{N}$ where $T_{N}$ is a Toeplitz band matrix.

In this paper the three models under consideration are :

(i) Hermitian Toeplitz band matrices. The model consists of
$N$-dimensional random Hermitian matrices $T_{N}=(\eta_{ij}\,
a_{i-j})_{i,j=1}^{N}$ in Eq.\,(\ref{bandtmatrices}). We assume that
$\mbox{Re}\,a_{j}=\mbox{Re}\,a_{-j}$ and
$\mbox{Im}\,a_{j}=-\mbox{Im}\,a_{-j}$ for $j=1,2,\cdots$, and
$\{a_{0}\}\cup \{\mbox{Re}\,a_{j},\mbox{Im}\,a_{j}\}_{j\in
\mathbb{N}}$ is a sequence of independent real random variables such
that \be \label{hermitiantoeplitz1}\mathbb{E}[a_{j}]=0, \ \
\mathbb{E}[|a_{j}|^{2}]=1 \ \ \textrm{for}\,\ \ j\in \mathbb{Z}\ee
 \noindent and further \be\label{hermitiantoeplitz2}\sup\limits_{ j\in \mathbb{Z}}
\mathbb{E}[|a_{j}|^{k}]=C_{k}<\iy\ \ \  \textrm{for} \ \ \ k\in
\mathbb{N}.\ee Notice that \begin{equation*} \sup\limits_{ j\in
\mathbb{Z}} \mathbb{E}[|a_{j}|^{k}]<\iy \Longleftrightarrow
\sup\limits_{ j\in \mathbb{Z}} \{\mathbb{E}[|\textrm{Re}
\,a_{j}|^{k}], \mathbb{E}[|\textrm{Im}\, a_{j}|^{k}]\}<\iy
,\end{equation*} and the assumption (\ref{hermitiantoeplitz2}) shows
that when fixing $k$, these moments of the independent random
variables whose orders are not larger than $k$ can be controlled by
some constant only depending on $k$.

The case of  real symmetric  Toeplitz band matrices is very similar
to the Hermitian case, except that we now consider $N$-dimensional
real symmetric matrices $T_{N}=(\eta_{ij}\, a_{i-j})_{i,j=1}^{N}$ .

(ii) Symmetric Toeplitz band matrices. We assume that
$a_{j}=a_{-j}$ for $j=0,1,\cdots$, and $\{a_{j}\}_{j=0}^{\iy}$ is a
sequence of independent real random variables such that

\be\label{symmetrictoeplitz1} \mathbb{E}[a_{j}]=0, \ \
\mathbb{E}[|a_{j}|^{2}]=1 \ \ \textrm{for}\,\ \ j=0,1,\cdots,\ee

 \noindent and further \be\label{symmetrictoeplitz2}\sup\limits_{j\geq 0} \mathbb{E}[|a_{j}|^{k}]=C_{k}<\iy\ \ \
\textrm{for} \ \ \ k\in \mathbb{N}.\ee

(iii) Symmetric Hankel band matrices. Let $T_{N}=(\eta_{ij}\,
a_{i-j})_{i,j=1}^{N}$ be a Toeplitz band matrix and
$H_{N}=P_{N}T_{N}$ be the corresponding Hankel band matrix. We
assume that $\{a_{j}\}_{j\in \mathbb{Z}}$ is a sequence of
independent real random variables such that

\be\label{symmetrichankel1} \mathbb{E}[a_{j}]=0, \ \
\mathbb{E}[|a_{j}|^{2}]=1 \ \ \textrm{for}\,\ \ j\in \mathbb{Z}\ee

 \noindent and further \be\label{symmetrichankel2}\sup\limits_{ j\in \mathbb{Z}} \mathbb{E}[|a_{j}|^{k}]=C_{k}<\iy\ \
\  \textrm{for} \ \ \ k\in \mathbb{N}.\ee

We assume that $b_{N}$ grows either as 1) $b_{N}/N \rightarrow b$ as
$N\rightarrow \infty$, $b\in (0,1]$ (proportional growth), or as
  2)
 $b_{N}\ra \iy$ as $N\rightarrow \infty$, $b_{N}=o(N)$ (slow
 growth).
 Under these assumptions we will establish the limit  distributions $\gamma_{_{T}}(b)$ for
Toeplitz band matrices and $\gamma_{_{H}}(b)$ for Hankel band
matrices. In the case $b>0$, we  prove that the distribution of
eigenvalues for Toeplitz (Hankel) band matrices converges almost
surely to $\gamma_{_{T}}(b)$ ($\gamma_{_{H}}(b)$), the moments of
which  are the sum of some integrals depending on $b$. For $b=0$, we
obtain that  $\gamma_{_{T}}(0)$ is the standard Gaussian
distribution and $\gamma_{_{H}}(0)$ is the distribution
$|x|\,\exp(-x^{2})$. In the case $b=0$, with the addition of
$b_{N}\geq C\,N^{\frac{1}{2}+\epsilon_{0}}$ for some positive
constants $\epsilon_{0}$ and $C$, we also prove
 almost sure convergence.

 The plan of the remaining part of our paper is the following. Sections 2 and 3 are devoted to the statements and proofs
 of the main theorems for band matrices with proportionally growing  and slowly growing bandwidths $b_{N}$, respectively.
In Section 4, the fourth moments of the limit distributions
$\gamma_{_{T}}(b)$ and $\gamma_{_{H}}(b)$ are calculated,  which
shows their difference for different $b$'s.

\section{Proportional Growth of $b_{N}$}
\setcounter{equation}{0} In order to calculate the moments of the
limit distribution, we review some basic combinatorical concepts.

\begin{definition} Let the set $[n]=\{1,2,\cdots,n\}$.

(1) We call $\pi=\{V_{1},\cdots,V_{r}\}$ a partition of $[n]$ if the
blocks  $V_{j} \,(1\leq j \leq r)$ are pairwise disjoint, non-empty
subsets of $[n]$ such that $[n]=V_{1}\cup\cdots\cup V_{r}$. The
number of blocks of $\pi$ is denoted by $|\pi|$, and the number of
the elements of $V_{j}$ is denoted by $|V_{j}|$.

(2) Without loss of generalization, we assume that
$V_{1},\cdots,V_{r}$ have been arranged such that
$s_{1}<s_{2}<\cdots<s_{r}$, where $s_{j}$ is the smallest number of
$V_{j}$. Therefore we can define the projection $\pi (i)=j$ if $i$
belongs to the block $V_{j}$; furthermore for two elements $p,q$ of
$[n]$ we write $p\thicksim_{\pi} q$ if $\pi (p)=\pi (q)$.

(3) The set of all partitions of $[n]$ is denoted by
$\mathcal{P}(n)$, and the subset consisting of all pair partitions,
i.e. all $|V_{j}|=2$, $1\leq j \leq r$, is denoted by
$\mathcal{P}_{2 }(n)$. The subset of $\mathcal{P}_{2 }(n)$
consisting of such pair partitions that each contains exactly one
even number and one odd number is denoted by $\mathcal{P}^{1}_{2
}(n)$. Note that $\mathcal{P}_{2 }(n)$ is an empty set if $n$ is
odd.
\end{definition}

 We can formulate our results for Hankel and
Toeplitz band matrices with the proportional growth of $b_{N}$ as
follows. By convention, $I_{B}$  represents the characteristic
function of a set $B$.

\begin{theorem}\label{tpg}Let $T_{N}$ be either a Hermitian  ((\ref{hermitiantoeplitz1})--(\ref{hermitiantoeplitz2}))
or real symmetric
((\ref{symmetrictoeplitz1})--(\ref{symmetrictoeplitz2})) Toeplitz
random band matrix, where $b_{N}/N \rightarrow b$ as $N\rightarrow
\infty$, $b\in (0,1]$. Take the normalization
$X_{N}=T_{N}/\sqrt{(2-b)b N}$, then $\mu_{X_{N}}$ converges almost
surely to a symmetric probability distribution $\gamma_{_{T}}(b)$
which is determined by its even moments \be
\label{bandtoeplitz:moment}m_{2k}(\gamma_{_{T}}(b))=\frac{1}{(2-b)^{k}}\sum_{\pi\in
\mathcal{P}_{2 }(2k)}\int_{[0,1]\times
[-1,1]^{k}}\prod_{j=1}^{2k}I_{[0,1]}(x_{0}+b\sum_{i=1}^{j}\epsilon_{\pi}(i)\,x_{\pi(i)})
\prod_{l=0}^{k}\mathrm{d}\, x_{l}\ee  where $\epsilon_{\pi}(i)=1$ if
$i$ is the smallest number of $\pi^{-1}(\pi(i))$; otherwise,
$\epsilon_{\pi}(i)=-1$.

\end{theorem}

\begin{theorem}\label{hpg}
Let $H_{N}$ be a real symmetric
((\ref{symmetrichankel1})--(\ref{symmetrichankel2})) Hankel random
band matrix, where $b_{N}/N \rightarrow b$ as $N\rightarrow \infty$,
$b\in (0,1]$. Take the normalization $Y_{N}=H_{N}/\sqrt{(2-b)b N}$,
then $\mu_{Y_{N}}$ converges almost surely to a symmetric
probability distribution $\gamma_{_{H}}(b)$ which is determined by
its even moments \be
\label{bandhankel:moment}m_{2k}(\gamma_{_{H}}(b))=\frac{1}{(2-b)^{k}}
\sum_{\pi\in \mathcal{P}^{1}_{2 }(2k)}\int_{[0,1]\times
[-1,1]^{k}}\prod_{j=1}^{2k}I_{[0,1]}(x_{0}-b\sum_{i=1}^{j}(-1)^{i}\,x_{\pi(i)})
\prod_{l=0}^{k}\mathrm{d}\, x_{l}. \ee
\end{theorem}

Let us first give two basic lemmas about traces of Toeplitz and
Hankel band matrices. Although their proofs are simple, they are
very useful in treating random matrix models closely related to
Toeplitz matrices.

\begin{lemma}\label{toeplitzlemma}
For a Toeplitz band matrix $T=(\eta_{ij}\, a_{i-j})_{i,j=1}^{N}$
with the bandwidth $b_{N}$ where $a_{-N+1}, \cdots,a_{N-1}$ are
 complex numbers, we have the trace formula
 \be \label{basic:lem1}
 \emph{tr}(T^{k})=\sum_{i=1}^{N}\,\sum_{j_{1},\cdots,j_{k}=-b_{N}}^{b_{N}}
\prod_{l=1}^{k}a_{j_{l}}\prod_{l=1}^{k}I_{[1,N]}(i+\sum_{q=1}^{l}j_{q})
\ \large{\delta}_{0,\sum\limits_{q=1}^{k}j_{q}}, k\in \mathbb{N}.
 \ee
\end{lemma}

\begin{proof}
For the standard basis $\{e_{1}, \cdots, e_{N}\}$ of the Euclidean
space $\mathbb{R}^{N}$, we have
\[T\, e_{i}=\sum_{j=0}^{b_{N}}a_{-j}
B^{j}\,e_{i}+ \sum_{j=1}^{b_{N}}a_{j} F^{j}\,e_{i}=
\sum_{j=-b_{N}}^{b_{N}}a_{j} I_{[1,N]}(i+j)\, e_{i+j}.\] Repeating
$T$'s effect on the basis, we have
\[T^{k}\, e_{i}=
\sum_{j_{1},\cdots,j_{k}=-b_{N}}^{b_{N}}
\prod_{l=1}^{k}a_{j_{l}}\prod_{l=1}^{k}I_{[1,N]}(i+\sum_{q=1}^{l}j_{q})
\,e_{i+\sum\limits_{q=1}^{k}j_{q}}.\
\]
By $\mbox{tr}(T^{k})=\sum\limits_{i=1}^{k}e^{t}_{i}\, T^{k}
\,e_{i}$, we complete the proof.
\end{proof}

\begin{lemma}\label{hankellemma}
Given a Toeplitz band matrix $T=(\eta_{ij}\, a_{i-j})_{i,j=1}^{N}$
with the bandwidth $b_{N}$ where $a_{-N+1}, \cdots,a_{N-1}$ are
 real numbers, let $H=PT$ be the corresponding Hankel band matrix where
$P=(\delta_{i-1,N-j})_{i,j=1}^{N}$. We have the trace formula

$\emph{tr}(T^{k})$=
 \be \label{basic:lem2}
 \begin{cases} \sum\limits_{i=1}^{N}\,\sum\limits_{j_{1},\cdots,j_{k}=-b_{N}}^{b_{N}}
\prod\limits_{l=1}^{k}a_{j_{l}}\prod\limits_{l=1}^{k}I_{[1,N]}(i-\sum\limits_{q=1}^{l}(-1)^{q}j_{q})
\ \large{\delta}_{0,\sum\limits_{q=1}^{k}(-1)^{q}j_{q}},&\text{$k$ even,}\\
\sum\limits_{i=1}^{N}\,\sum\limits_{j_{1},\cdots,j_{k}=-b_{N}}^{b_{N}}
\prod\limits_{l=1}^{k}a_{j_{l}}\prod\limits_{l=1}^{k}I_{[1,N]}(i-\sum\limits_{q=1}^{l}(-1)^{q}j_{q})
\
\large{\delta}_{2i-1-N,\sum\limits_{q=1}^{k}(-1)^{q}j_{q}},&\text{$k$
odd.}
\end{cases}
 \ee
\end{lemma}

\begin{proof}
Follow a similar procedure as in the proof of Lemma
\ref{toeplitzlemma}.
\end{proof}

We are now ready to prove the main results of this section.

\begin{proof}[Proof of Theorem \ref{tpg}] We prove the theorem by the following steps:

\textbf{Step 1. Calculation of the Moments}

Let \[m_{k,N}=\mathbb{E}[\int x^{k} \mu_{X_{N}}(d\, x)]=\frac{1}{N}
\mathbb{E}[\textrm{tr} (X_{N}^{k})].\] Using Lemma 2.4, we have
\begin{equation*}
 m_{k,N}=\frac{N^{-\frac{k}{2}-1}}{(2-b)^{\frac{k}{2}}b^{\frac{k}{2}}}\sum_{i=1}^{N}\,\sum_{j_{1},\cdots,j_{k}=-b_{N}}^{b_{N}}
\prod_{l=1}^{k}I_{[1,N]}(i+\sum_{q=1}^{l}j_{q}) \
\large{\delta}_{0,\sum\limits_{q=1}^{k}j_{q}}\mathbb{E}[\prod_{l=1}^{k}a_{j_{l}}].
 \end{equation*}
For \textbf{j}=$(j_{1},\cdots,j_{k})$, we construct a set of numbers
$S_{\textbf{j}}=\{|j_{1}|,\ldots,|j_{k}|\}$ with multiplicities.
Note that the random variables whose subscripts have different
absolute values are independent. If $S_{\textbf{j}}$ has one number
with multiplicity 1, by independence of the random variables, we
have $\mathbb{E}[\prod_{l=1}^{k}a_{j_{l}}]=0$. Thus the only
contribution to the $k$-th moment comes when each of
$S_{\textbf{j}}$ has multiplicity at least 2, which implies that
$S_{\textbf{j}}$ has at most $[\frac{k}{2}]$ distinct numbers. Once
we have specified the distinct numbers of $S_{\textbf{j}}$, the
subscripts $j_{1},\cdots,j_{k}$ are determined in at most
$2^{k}([\frac{k}{2}])^{k}$ ways. By independence and the assumptions
(\ref{hermitiantoeplitz2}) and (\ref{symmetrictoeplitz2}), we find
\[m_{k,N}=O((b_{N})^{-\frac{k}{2}+[\frac{k}{2}]})=O(N^{-\frac{k}{2}+[\frac{k}{2}]}).\]
Therefore, for odd $k$
\[\lim_{N\ra \iy}m_{k,N}=0.\]
It suffices to deal with $m_{2k,N}$. If each of
$S_{\textbf{j}}=\{|j_{1}|,\cdots,|j_{2k}|\}$ (here
\textbf{j}=$(j_{1},\cdots,j_{2k})$) has multiplicity at least 2 but
one of which at least 3, then $S_{\textbf{j}}$ has at most $k-1$
distinct numbers. Thus, the contribution of such terms to $m_{2k,N}$
is $O(N^{-1})$. So it suffices to consider all pair partitions of
$[2k]=\{1,2,\cdots,2k\}$. That is, for $\pi \in
\mathcal{P}_{2}(2k)$, if $p\thicksim_{\pi}q$, then it is always the
case  that $j_{p}=j_{q}$ or $j_{p}=-j_{q}$.  Under the condition
\[\sum\limits_{q=1}^{2k}j_{q}=0\] according to (\ref{basic:lem1}), considering the main
contribution to the trace,  we should take $j_{p}=-j_{q}$ ;
otherwise, there exists $p_{0}, q_{0}\in [2k]$ such that
$$j_{p_{0}}=j_{q_{0}}=\frac{1}{2}(j_{p_{0}}+j_{q_{0}}-\sum\limits_{q=1}^{2k}j_{q}).$$
We can choose other $k-1$ distinct numbers, which determine
$j_{p_{0}}=j_{q_{0}}$. This shows that there is a loss of at least
 one degree of freedom and the contribution of such terms is $O(N^{-1}).$ Thus
$a_{j_{p}}$ and $a_{j_{q}}$ are conjugate. So we can write \be
\label{integral:sum} m_{2k,N}=o(1)+
\frac{N^{-k-1}}{(2-b)^{k}b^{k}}\sum_{\pi \in \mathcal{P}_{2}(2k)
}\sum_{i=1}^{N}\,\sum_{j_{1},\cdots,j_{k}=-b_{N}}^{b_{N}}
\prod_{l=1}^{2k}I_{[1,N]}(i+\sum_{q=1}^{l}
\epsilon_{\pi}(q)j_{\pi(q)}). \ee

Follow a similar argument from \cite{bdj,bmp}: for fixed $\pi \in
\mathcal{P}_{2}(2k)$,
$$
\frac{1}{N^{k+1}}\sum_{i=1}^{N}\,\sum_{j_{1},\cdots,j_{k}=-b_{N}}^{b_{N}}
\prod_{l=1}^{2k}I_{[1,N]}(i+\sum_{q=1}^{l}
\epsilon_{\pi}(q)j_{\pi(q)}),$$ i.e.
$$
\frac{1}{N^{k+1}}\sum_{i=1}^{N}\,\sum_{j_{1},\cdots,j_{k}=-b_{N}}^{b_{N}}
\prod_{l=1}^{2k}I_{[\frac{1}{N},1]}(\frac{i}{N}+\sum_{q=1}^{l}
\epsilon_{\pi}(q)\frac{j_{\pi(q)}}{N}),$$ can be considered as a
Riemann sum of the definite integral
$$\int_{[0,1]\times [-b,b]^{k}}\prod_{l=1}^{2k}I_{[0,1]}(x_{0}+\sum_{q=1}^{l}
\epsilon_{\pi}(q)x_{\pi(q)}) \prod_{l=0}^{k}\mathrm{d}\, x_{l}.$$
Given $\pi \in \mathcal{P}_{2}(2k)$, as $N\ra \iy$, each term in
(\ref {integral:sum}) can be treated as an integral. Thus we have
$m_{2k}(\gamma_{_{T}}(b))=\lim\limits_{N\ra\iy}m_{2k,N}$ is the
representation of the form (\ref{bandtoeplitz:moment}).

\textbf{Step 2. Carleman's Condition}

Obviously, from (\ref{bandtoeplitz:moment}) we have
\[m_{2k}(\gamma_{_{T}}(b))\leq \frac{1}{(2-b)^{k}}\sum_{\pi \in
\mathcal{P}_{2}(2k)}\int_{[0,1]\times [-1,1]^{k}}1
\prod_{l=0}^{k}\mathrm{d}\, x_{l}=\frac{2^{k}}{(2-b)^{k}}(2k-1)!!.
\]
By Carleman's theorem (see \cite{feller}), the limit distribution
$\gamma_{_{T}}(b)$ is uniquely determined by the moments.

\textbf{Step 3. Almost Sure Convergence}

It suffices to show that \be
\sum_{N=1}^{\infty}\frac{1}{N^{4}}\mathbb{E}[(\textrm{tr}
(X_{N}^{k})-\mathbb{E}[\textrm{tr} (X_{N}^{k})])^{4}]<\infty\ee for
each fixed $k$. Indeed, by Lemma \ref{toeplitzlemma}, we have
$$\textrm{tr} (X_{N}^{k})=\frac{1}{N^{\frac{k}{2}}}\sum_{p_{0},\textbf{p}} A[p_{0};\textbf{p}]$$
where
$$A[p_{0};\textbf{p}]=\prod_{l=1}^{k}a_{p_{l}}\prod_{l=1}^{k}I_{[1,N]}(p_{0}+\sum_{q=1}^{l}p_{q})
\large{\delta}_{0,\sum\limits_{q=1}^{k}p_{q}},$$
$\textbf{p}=(p_{1},\ldots,p_{k})$, and the summation
$\sum_{p_{0},\textbf{p}}$
 runs over all possibilities
that $\textbf{p} \in \{-b_{N},\ldots,b_{N}\}^{k}$ and $p_{0} \in
\{1,\ldots,N\}$. Thus,
\begin{align} \label{fourtermsum}\frac{1}{N^{4}}&\mathbb{E}[(\textrm{tr}
(X_{N}^{k})-\mathbb{E}[\textrm{tr} (X_{N}^{k})])^{4}]\nonumber\\
&=\frac{(2-b)^{-2k}b^{-2k}}{N^{4+2k}}\sum_{i_{0},j_{0},s_{0},t_{0}\atop
\textbf{i},\textbf{j},\textbf{s},\textbf{t}}
\mathbb{E}[\prod_{\textbf{p}\in
\{\textbf{i},\textbf{j},\textbf{s},\textbf{t}\}}(A[p_{0};\textbf{p}]-\mathbb{E}[A[p_{0};\textbf{p}]])]
 \end{align}
where the summation $\sum_{i_{0},j_{0},s_{0},t_{0}\atop
\textbf{i},\textbf{j},\textbf{s},\textbf{t}}$ runs over all
possibilities that $\textbf{i},\textbf{j},\textbf{s},\textbf{t} \in
\{-b_{N},\ldots,b_{N}\}^{k}$ and $i_{0},j_{0},s_{0},t_{0} \in
\{1,\ldots,N\}$.

For $\textbf{i}=(i_{1},\ldots,i_{k})$, as in \textbf{Step} 1 we
still construct a set of numbers
$S_{\textbf{i}}=\{|i_{1}|,\ldots,|i_{k}|\}$ with  multiplicities.
Obviously, if
 one of $S_{\textbf{i}}, \ldots, S_{\textbf{t}}$, for example, $S_{\textbf{i}}$ does not have any
 number coincident with numbers of the other three sets, then the
 term in (\ref{fourtermsum}) equals 0 by independence. Also, if the union  $S=S_{\textbf{i}}\cup \cdots \cup
 S_{\textbf{t}}$ has one number with  multiplicity 1,  the
 term in (\ref{fourtermsum}) equals 0.

Now, let us estimate the non-zero term in (\ref{fourtermsum}).
Assume that $S$ has $p$ distinct numbers with multiplicities
$\nu_{1}, \ldots, \nu_{p}$, subject to the constraint $\nu_{1}+
\ldots+ \nu_{p}=4k$. If $p\leq 2k-2$, by
(\ref{hermitiantoeplitz1})--(\ref{symmetrictoeplitz2}), the
contribution of the terms corresponding to $S$ is bounded by
$$C_{k}\frac{(2-b)^{-2k}b^{-2k}}{N^{4+2k}} N^{4} \,b_{N}^{2k-2}\leq C_{k,b}\,N^{-2}$$
for some constants $C_{k}$ and $C_{k,b}$.

For the case of $p=2k-1$, we have two numbers shared by three of
four sets $S_{\textbf{i}}, \ldots, S_{\textbf{t}}$ each or one
number shared by the four sets. It is not hard to see that there
always exists one of $S_{\textbf{i}}, \ldots, S_{\textbf{t}}$, for
example, $S_{\textbf{i}}$, which has one number with multiplicity 1,
denoted by $|i_{q_{0}}|$ for some $q_{0}\in \{1,\ldots,k\}$.
Consequently, under the constraint $$\sum_{q=1}^{k}i_{q}=0$$
according to (\ref{basic:lem1}), we have
$i_{q_{0}}=i_{q_{0}}-\sum\limits_{q=1}^{k}i_{q}$. Hence we can
choose the other $2k-2$ distinct numbers of $S$, which determine the
subscript $i_{q_{0}}$. This shows that there is a loss of at least
one degree of freedom and  the contribution of the terms
corresponding to $S$ is bounded by $C_{k,b}\,N^{-2}$.

The case of $p=2k$ implies each number in $S$ occurs exactly two
times.  Thus there exist two of $S_{\textbf{i}}, \ldots,
S_{\textbf{t}}$, for example, $S_{\textbf{i}}$ and $S_{\textbf{j}}$,
which share one number with one of $S_{\textbf{s}}$ and
$S_{\textbf{t}}$ respectively. Consequently,  under the constraints
\[\sum\limits_{q=1}^{k}i_{q}=0 \ \ \mbox{and}\ \  \sum\limits_{q=1}^{k}j_{q}=0\] according to (\ref{basic:lem1}), there is a
loss of at least two degrees of freedom and  the contribution of the
terms corresponding to $S$ is bounded by $C_{k,b}\,N^{-2}$.

Therefore,
\[\frac{1}{N^{4}} \mathbb{E}[(\textrm{tr}
(X_{N}^{k})-\mathbb{E}[\textrm{tr} (X_{N}^{k})])^{4}]\leq
C_{k,b}\,N^{-2}\] and almost sure convergence is proved.
Consequently, the proof of Theorem \ref{tpg} is complete.
\end{proof}

\begin{proof}[Proof of Theorem \ref{hpg}]
Using Lemma \ref{hankellemma} and our  assumptions, the proof is
very similar to that of of Theorem \ref{tpg}. Here we don't repeat
the process. Note that in the representation of even moments in
Eq.\,(\ref{bandhankel:moment}) the pair partition $\pi \in
\mathcal{P}^{1}_{2}(2k)$ comes from the constraint \be
\sum\limits_{q=1}^{k}(-1)^{q}j_{q}=0\ee in the form of the trace for
Hankel matrices in Lemma \ref{hankellemma}. \end{proof}

\section{Slow Growth of $b_{N}$}
\setcounter{equation}{0}

 We give our results for Hankel and
Toeplitz band matrices with slow growth of the bandwidth as follows.

\begin{theorem}\label{tsg}Let $T_{N}$ be either a Hermitian  ((\ref{hermitiantoeplitz1})--(\ref{hermitiantoeplitz2}))
or real symmetric
((\ref{symmetrictoeplitz1})--(\ref{symmetrictoeplitz2})) Toeplitz
random band matrix, where $b_{N}\rightarrow \infty$ but $b_{N}/N
\rightarrow 0$ as $N\rightarrow \infty$. Take the normalization
$X_{N}=T_{N}/\sqrt{2\, b_{N}}$, then $\mu_{X_{N}}$ converges weakly
 to the standard Gaussian distribution.
 In addition, if there exist some positive constants $\epsilon_{0}$ and $C$ such that
 $b_{N}\geq C\,N^{\frac{1}{2}+\epsilon_{0}}$, then $\mu_{X_{N}}$ converges almost surely to the standard Gaussian distribution.
\end{theorem}

\begin{theorem}\label{hsg}
Let $H_{N}$ be a real symmetric
((\ref{symmetrichankel1})--(\ref{symmetrichankel2})) Hankel random
band matrix, where $b_{N}\rightarrow \infty$ but $b_{N}/N
\rightarrow 0$ as $N\rightarrow \infty$. Take the normalization
$Y_{N}=H_{N}/\sqrt{2\, b_{N}}$, then $\mu_{Y_{N}}$ converges weakly
 to the distribution $f(x)=|x|\exp(-x^{2})$.
 In addition, if there exist positive constants $\epsilon_{0}$ and $C$ such that
 $b_{N}\geq C\,N^{\frac{1}{2}+\epsilon_{0}}$, then $\mu_{Y_{N}}$ converges almost surely to
 the distribution $f(x)=|x|\exp(-x^{2})$.
\end{theorem}

\begin{remark}
For real symmetric palindromic Toeplitz Matrices, i.e. real
symmetric Toeplitz matrices under extra conditions:
$a_{j-1}=a_{N-j}, 0<j<N$, Massey, Miller and
 Sinsheimer \cite{mms} have obtained the Gaussian normal distribution for eigenvalues. And for random reverse circulant matrices, i.e., Hankel
matrices under extra conditions: $a_{-j}=a_{N-j}, 1<j<N$, Bose and
Mitra \cite{bm} obtained the same distribution of eigenvalues
$f(x)=|x|\,\exp(-x^{2})$ as in Theorem \ref{hsg}.
\end{remark}

Since proofs of both theorems are very similar, we prove only the
Toeplitz case.

\begin{proof}[Proof of Theorem \ref{tsg}]
The proof is  quite similar to that  of Theorem \ref{tpg} and some
details will be omitted. We will lay a strong emphasis on the
derivation of the Gaussian distribution. Here the slow growth of
$b_{N}$ leads to an easy calculation of the complicated integrals.
 We now complete the proof of the Gaussian law by showing

(1) $m_{k,N}=\frac{1}{N} \mathbb{E}[\textrm{tr} (X_{N}^{k})]$
converges to the $k$-th moment of the standard Gaussian
distribution.

(2) Assume that $b_{N}\geq C\,N^{\frac{1}{2}+\epsilon_{0}}$ for some
positive constants $\epsilon_{0}$ and $C$. For each fixed $k$, \be
\sum_{N=1}^{\infty}\frac{1}{N^{4}} \mathbb{E}[(\textrm{tr}
(X_{N}^{k})-\mathbb{E}[\textrm{tr} (X_{N}^{k})])^{4}]<\infty.\ee

By Lemma \ref{toeplitzlemma} and our assumptions
(\ref{hermitiantoeplitz1})--(\ref{symmetrictoeplitz2}), it follows
that
\[m_{k,N}=O((b_{N})^{-\frac{k}{2}+[\frac{k}{2}]}).\]
Since  $b_{N}\rightarrow \infty$ as $N\rightarrow \infty$, for odd
$k$
\[\lim_{N\ra \iy}m_{k,N}=0.\]
It suffices to deal with $m_{2k,N}$. However, again by Lemma
\ref{toeplitzlemma} and our assumptions, the contribution with the
exception of all pair partitions to $m_{2k,N}$ is $O(b_{N}^{-1})$.
So it suffices to consider all pair partitions of
$[2k]=\{1,2,\cdots,2k\}$. That is, for $\pi \in
\mathcal{P}_{2}(2k)$, if $p\thicksim_{\pi}q$, then it is always the
case  that $j_{p}=j_{q}$ or $j_{p}=-j_{q}$.  Under the condition
\[\sum\limits_{q=1}^{2k}j_{q}=0\] according to (\ref{basic:lem1}), considering the main
contribution to the trace,  we should take $j_{p}=-j_{q}$
(otherwise, there is a loss of at least one degree of freedom). Thus
$a_{j_{p}}$ and $a_{j_{q}}$ are conjugate. So we can write \be
\label{integral:sum2} m_{2k,N}=o(1)+
\frac{N^{-1}}{2^{k}b_{N}^{k}}\sum_{\pi \in \mathcal{P}_{2}(2k)
}\sum_{i=1}^{N}\,\sum_{j_{1},\cdots,j_{k}=-b_{N}}^{b_{N}}
\prod_{l=1}^{2k}I_{[1,N]}(i+\sum_{q=1}^{l}
\epsilon_{\pi}(q)j_{\pi(q)}). \ee

Note that the constraints \be 1\leq i+\sum_{q=1}^{l}
\epsilon_{\pi}(q)j_{\pi(q)}\leq N,\ \ 1\leq l\leq 2k, \nonumber\ee
i.e.
 \be \frac{1}{N}\leq \frac{i}{N}+\sum_{q=1}^{l}
\epsilon_{\pi}(q)\frac{j_{\pi(q)}}{N}\leq 1,\ \ 1\leq l\leq 2k.
\nonumber\ee Since $$|\sum_{q=1}^{l}
\epsilon_{\pi}(q)\frac{j_{\pi(q)}}{N}|\leq l
\frac{b_{N}}{N}\rightarrow 0$$ as $N \rightarrow \iy$, we have
$$I_{[1,N]}(i+\sum_{q=1}^{l}
\epsilon_{\pi}(q)j_{\pi(q)})=1+o(1).$$Hence \be
m_{2k,N}=o(1)+(2k-1)!!\,(\frac{2b_{N}+1}{2b_{N}})^{k}(1+o(1))^{2k}\ee
converges to $(2k-1)!!$ as $N \rightarrow \iy$. Assertion (1) is
then proved.

We now prove (2). The same argument as in the proof of almost sure
convergence in Theorem \ref{tpg} shows that
\[\frac{1}{N^{4}} \mathbb{E}[(\textrm{tr}
(X_{N}^{k})-\mathbb{E}[\textrm{tr}
(X_{N}^{k})])^{4}]=O(\frac{1}{b_{N}^{2}}).\] Since $b_{N}\geq
C\,N^{\frac{1}{2}+\epsilon_{0}}$, we have
\[\frac{1}{N^{4}} \mathbb{E}[(\textrm{tr}
(X_{N}^{k})-\mathbb{E}[\textrm{tr} (X_{N}^{k})])^{4}]\leq
C_{k}\frac{1}{N^{1+2\epsilon_{0}}},\] where $C_{k}$ is a constant
depending on $k$ only. This completes the proof of assertion (2).

The proof of Theorem \ref{tsg} is then complete.
\end{proof}

\section{First Four Moments of $\gamma_{_{T}}(b)$ and $\gamma_{_{H}}(b)$}
\setcounter{equation}{0}

We will compute the second and fourth moments of $\gamma_{_{T}}(b)$
and $\gamma_{_{H}}(b)$. From the fourth moment we can read that the
$\gamma_{_{T}}(b)$'s are different for different $b$'s, the
$\gamma_{_{H}}(b)$'s different for different $b\in [0,\frac{1}{2}]$
and the $\gamma_{_{H}}(b)$'s different for different $b\in
[\frac{1}{2},1]$. However, for the $2k$-th moments ($k \geq 3$), the
integrals on the right-hand sides of (\ref{bandtoeplitz:moment}) and
(\ref{bandhankel:moment}) are in general different for different
pair partitions, which makes difficult the explicit calculations of
the higher moments.

Observe that for $\pi\in \mathcal{P}^{1}_{2}(2k)$, the corresponding
integrals on the right-hand sides of (\ref{bandtoeplitz:moment}) and
(\ref{bandhankel:moment}) are in fact the same. Therefore, it is
sufficient to calculate the integral in the Toeplitz case for every
pair partition. For the pair partition $\pi \in\mathcal{P}_{2}(2k)$,
we introduce the symbol
\[p_{\pi}(b)=\int_{[0,1]\times
[-1,1]^{k}}\prod_{j=1}^{2k}I_{[0,1]}(x_{0}+b\sum_{i=1}^{j}\epsilon_{\pi}(i)\,x_{\pi(i)})
\prod_{l=0}^{k}\mathrm{d}\, x_{l}.\]

For $k$=1 it is easy to obtain the second moments
\[m_{2}(\gamma_{_{T}}(b))=1,\ \ \ m_{2}(\gamma_{_{H}}(b))=1.\]
Thus the fourth moment is the first ``free" moment in seeing the
shape of the distribution. When $k=2$, a direct calculation (this
can be checked using Mathematica) for all pair partitions
$$\pi_{1}=\{\{1,2\},\{3,4\} \},
 \pi_{2}=\{ \{1,4\},\{2,3\} \},
\pi_{3}=\{\{1,3\},\{2,4\}\}$$ of $\mathcal{P}_{2}(4)$ shows

\be \label{integral1&2}p_{\pi_{1}}(b)=p_{\pi_{2}}(b)=
 \begin{cases} \frac{2}{3}(6-5b),\ \ b\in [0,\frac{1}{2}];\\
\frac{-1+6b-2b^{3}}{3b^{2}},\ \ b\in (\frac{1}{2},1]
\end{cases}
\ee
 and
\be\label{integral3}p_{\pi_{3}}(b)=
 \begin{cases} 4(1-b),\ \ \hspace{1.3cm} b\in [0,\frac{1}{2}];\\
\frac{2(-1+6b-6b^{2}+2b^{3})}{3b^{2}},\ \ b\in (\frac{1}{2},1].
\end{cases}
\ee
 Thus, for $b\in [0,\frac{1}{2}]$,
\[m_{4}(\gamma_{_{T}}(b))=\frac{4(9-8b)}{3(2-b)^{2}}\]
which strictly decreases on $[0,\frac{1}{2}]$. When $b\in
(\frac{1}{2},1]$,
\[m_{4}(\gamma_{_{T}}(b))=\frac{4(-1+6b-3b^{2})}{3b^{2}(2-b)^{2}}\]
which also strictly decreases on $(\frac{1}{2},1]$.

Therefore, one knows that $m_{4}(\gamma_{_{T}}(b))$ strictly
decreases on $[0,1]$. Further, the distributions
$\gamma_{_{T}}(b)$'s are different. In particular,
$\gamma_{_{T}}(b)$ ($0<b\leq 1$) is indeed different from the normal
distribution. Note that for $b=1$ the fact that the limit
distribution is not Gaussian has been observed in \cite{bcg,bdj,hm}.

We turn to the Hankel-type distribution $\gamma_{_{H}}(b)$. From
(\ref{integral1&2}), one obtains that for $b\in [0,\frac{1}{2}]$
\[m_{4}(\gamma_{_{H}}(b))=\frac{4(6-5b)}{3(2-b)^{2}}\] which strictly
decreases on $[0,\frac{1}{2}]$ while for $b\in (\frac{1}{2},1]$,
\[m_{4}(\gamma_{_{H}}(b))=\frac{2(-1+6b-2b^{2})}{3b^{2}(2-b)^{2}}\]
which strictly increases on $(\frac{1}{2},1]$. Thus, according to
the fourth moments, we only know that the $\gamma_{_{H}}(b)$'s
($0\leq b \leq \frac{1}{2}$) are different and the
$\gamma_{_{H}}(b)$'s ($\frac{1}{2}\leq b \leq 1$) are different.

\vspace{.2cm}

\noindent\textbf{An added note.} After the paper as submitted we
learned from the Associate Editor and the referee about a related
preprint ``Limiting Spectral Distribution of Some Band Matrices" by
Basak and Bose at the site
http://www.isical.ac.in/~statmath/html/publication/techreport.html.
Although their paper and ours contain the same results for Hankel
and Toeplitz band matrices, the former assumes less integrability on
the entries of a matrix, allows more general ``rates" for the
bandwidth, and also covers more ensembles of ``structured matrices"
related to Toeplitz matrices. On the other hand, our paper covers
Hermitian Toeplitz matrices and  gives low order moment
calculations. Besides, our paper gives a different method for
analyzing Toeplitz matrices by treating them as a linear combination
of deterministic matrices with independent coefficients. This method
 can be used to derive
other results, including those that deal with semicircle law.

\section*{Acknowledgements}
The authors thank the Associate Editor for valuable comments, the
referees for very helpful comments and careful writing instruction
on the draft of this paper, and Steven Miller for information on his
paper.


\end{document}